\documentclass[reqno]{amsart} 

\usepackage{amsmath}
\usepackage{amssymb}
\usepackage{amscd}
\usepackage[all]{xy}
\CompileMatrices
\input{xypic}

\theoremstyle{plain}
\newtheorem{thm}{Theorem}[section]
\newtheorem{prop}[thm]{Proposition}

\newtheorem{cor}[thm]{Corollary}
\newtheorem{cor*}[thm]{Corollary}
\newtheorem*{thm*}{Theorem}
\newtheorem*{prop*}{Proposition}
\theoremstyle{definition}
\newtheorem{defn}[thm]{Definition}

\theoremstyle{remark}

\newtheorem{remark}{Remark}

\newcommand\vv{\,\vert\,}
\newcommand\sub{\subseteq}

\newcommand\mc[1]{\mathcal #1}

\newcommand\dph{\text{depth}\,}
\newcommand\as{\text{Ass}\,}
\newcommand\rk{\text{rk }}
\newcommand\an{\text{ann}\,}
\newcommand\p[1]{\frak p_{#1}}
\newcommand\fg{finitely generated }

\newcommand\res{\text{res}}

\title{Depth and Associated Primes in Group Cohomology} 
\author{James A. Schafer}
\address{\begin{flushleft}\quad Department of Mathematics \\
\quad University of Maryland \\
\quad College Park, Maryland 20742
\end{flushleft}}

\date{\today}

\begin{document}
\maketitle

\section*{Introduction} If $R$ is a Noetherian local ring and $M$ is a \fg $R$-module then the depth of $M$ is defined to be the length of a maximal $M$-sequence in the maximal ideal, $\frak m$, of $R$.  Serre demonstrated that $$\dph M\leq \omega_M=\text{ the minimum dimension of an associated prime of } M,$$ that is the minimum dimension of a prime ideal equal to the annihilator of an element of $M$.  This result was easily translated to $*$-local $k$-algebras $A$, i.e., graded $k$-algebras with a unique homogeneous maximal ideal.  There are easy examples of strict inequality. After many calculations of the cohomology rings of finite groups, Carlson \cite{C}, asked whether the above inequality was true if $A=H^*(G,k)$ where $G$ was a finite group and $k$ was a field whose charateristic divided the order of $G$.  Carlson showed it was true if $\dim H^*(G,k)=2$. Duflot \cite{D} had shown $\dph H^*(G,k)\geq p$-rank of the center of a Sylow $p$-subgroup of $G$. A result for $p$-groups by Green \cite{G} and later extended by Kuhn \cite{K} to all finite groups showed that if the Duflot lower bound was an equality the conjecture was true.  In \cite{S} it is shown that if $\dim H^*(G,k)-\dph H^*(G,k)\leq 1$ the conjecture is true. To my knowledge nothing else is known.

It is a fact, \cite{N}, that $\dph H^*(C_GE,k)\geq\dph H^*(G,k)$ always. The aim of this paper to show that the Conner Conjecture is equivalent to the following: There exists an elementary abelian $p$-subgroup of $G$ of rank equal to $\omega_G=\omega_{H^*(G,k)}$ with $\dph H^*(C_GE,k)=\dph H^*(G,k)$.

It is also true, (\cite{B1}, \cite{Se}, \cite{W}), that if $P$ is an associated prime of $H^*(G,k)$, then
$$P=\frak p_E=\ker\{\res:H^*(G,k)\to H^*(E,k)/\text{rad}\}$$ for some elementary abelian $p$-subgroup of $G$.   We also show that if $\frak p_E=\an y$ and $\dim \frak p_E=\omega_G$ then $\frak P_E=\ker\{\res:H^*(C_GE,k)\to H^*(E,k)/\text{rad}\}$ is an associated prime of $H^*(C_GE,k)$ of dimension $\omega_G=\omega_{C_GE}$ and in fact $\frak P_E=\text{rad}(\an(\res\, y))$.

\section*{Notation and a useful result}

All rings will be commutative, N\"oetherian with identity. If $M$ is a \fg $R$-module, denote the associated primes of $M$ by $\as\,M$. $\mc AG$ will denote the elementary abelian $p$-subgroups of $G$ and $\mc AG^s$ those of rank $s$.
Unless otherwise stated $\res^G_H:H^*(G)\to H^*(H)$ for $H\sub G$ will be denoted by $\res_H$.

\medskip
The following useful result is easily extracted from Atiyah-MacDonald, \cite{AM}.

\begin{thm}\label{non} Let $f:R\to S$ be a monomorphism of rings and $P\in\as\,R$, then there exists $Q\in\as\, S$ with $f^{-1}Q=P$.\end{thm}
\begin{proof} Let $(0)=\bigcap Q_i$ be a minimal primary decomposition of $(0)\sub S$ with primes $P_i=\text{rad}\, Q_i$. Then $\as\,S=\{P_i\}$.  Since the contraction of a prmary ideal is primary and the contraction of the radical of an ideal is the radical of the contraction, $(0)=f^{-1}(0)=\bigcap f^{-1}Q_i$ is a primary decompositon of $(0)\sub R$ with primes $\{f^{-1}P_i\}$. Since any primary decompostion of an ideal can be refined to a minimal primary decomposition, $\as\,R\sub\{f^{-1}P_i\}$. 
\end{proof} 

\begin{remark} This does not mean that if $P\in\as\,S$ then $f^{-1}P\in\as\,R$ and in fact is false.\end{remark}

\section*{Special facts for $H^*(G,k)$}

From this point on $G$ will be a finite group, $k$ a field with trivial $G$-action and $p=\text{char}\,k$ dividing the order of $G$. Denote $H^*(G,k)$ by $H^*(G)$.
\medskip 

The following result may be found in \cite{E}, corollaries 7.4.6 and 7.4.7.

\begin{thm}\label{Ev}(1) $H^*(G)$ is a \fg $k$-algebra, which is the same as saying $H^*(G)$ is N\"oetherian.
\par (2) If $H\sub G$ is a subgroup then $H^*(H)$ is \fg as a module over $H^*(G)$ via the restriction map $\res:H^*(G)\to H^*(H)$\end{thm}

What properties does $H^*(G)$ have concerning depth and associated primes have that an arbitrary \fg graded $k$-algebra may not.

\begin{thm} Let $N\sub G$ with $[G:N]$ relatively prime to $p$. If $\{x_i\}\sub H^*(G)$ is a finite set of homogeneous elements and $\{\text{res}_N\,x_i\}$ is a regular sequence in $H^*(N)$ then $\{x_i\}$ is a regular sequence in $H^*(G)$.\end{thm}
\begin{proof} In \cite{Kap} it is shown that for a non-negatively graded \fg $k$-algebra $R$, $\{x_i\}\sub R^+$ is a regular sequence if and only $k[x_1,\dots,x_n]$ is a polynomial subring of $R$ and $R$ is free over $k[x_1,\dots,x_n]$. If $\{\text{res}_N\,x_i\}$ is a regular sequence in $H^*(N)$, then $k[\res_N\,x_1,\dots,\res_Nx_n]$ is a polynomial algebra and $H^*(N)$ is free over $k[\res_Nx_1,\dots,\res_Nx_n]$. Hence $H^*(N)$ is free over $k[x_1,\dots,x_n]\sub H^*(G)$.
 
$H^*(G)\sub H^*(N)$ is a $H^*(G)$-summand of $H^*(N)$ therefore projective hence a free $k[x_1,\dots,x_n]$-module and $\{x_i\}$ is a regular sequence in $H^*(G)$.\end{proof}

\begin{thm}(Duflot,\,\cite{D})\label{Duf} Let $C\sub G$ be a central subgroup. If $\{x_i\}\sub H^*(G)$ is a finite set of homogeneous elements and $\{\text{res}_C\,x_i\}$ is a regular sequence in $H^*(C)$ then $\{x_i\}$ is a regular sequence in $H^*(G)$.\end{thm}

\begin{cor} Let $C\sub Z(P)\sub P\sub G$ where $P$ is a Sylow $p$-subgroup of $G$.  If $\{x_i\}\sub H^*(G)$ is a finite set of homogeneous elements and $\{\text{res}_C\,x_i\}$ is a regular sequence in $H^*(C)$ then $\{x_i\}$ is a regular sequence in $H^*(G)$.\end{cor}

\begin{cor} $\dph H^*(G)\geq p$-rank of the center of a Sylow $p$-subgroup.\end{cor}

\begin{defn} If $H\sub G$ and $E\in\mathcal AH$, let \
$$\frak p_E^H=\ker\{\pi\,\res:H^*(H)\to H^*(E)\to H^*(E)/\text{Rad}\}.$$
If $H=G$ denote $\frak p_E^G$ by $\frak p_E$
\end{defn}

\begin{thm}(Serre,\,\cite{Se},Wilkerson,\,\cite{W}) If $P$ is an associated prime of $H^*(G)$ then $P=\frak p_E$ for some $E\in\mc AG$.
\end{thm}

\begin{thm}(Quillen,\,\cite{Q})\label{Qu} $\frak p_E\sub \frak p_{E'}$ iff $E'$ is conjugate to a subgroup of $E$. Hence since minimal primes are associated primes, minimal primes in $H^*(G)$ correspond to conjugacy classes of maximal $p$-elementary abelian subgroups of $G$. \end{thm}

\begin{cor} Let \,$E$ be elementary $p$-abelian.  Then there exists $\tau_E\in\frak p_E$ such that $\tau_{E}\notin\frak p_{E'}$ for any elemenary abelian $E'$ not conjugate to a subgroup of $E$.\end{cor}
\begin{proof} Let $W=\{E'\in\mathcal A(G)\vv E'\,\text{not conjugate to a subgroup of}\,E\}$. Then $\frak p_E\not\subseteq \frak p_{E'}$ for any $E'\in W$.  By Prime Avoidence $\frak p_E\not\subseteq\bigcup_W\frak p_{E'}$.  Therefore there exists $\tau_E\in\frak p_E$ such that $\tau_{E}\notin\frak p_{E'}$ for any $E'\in W$. \end{proof} 

\begin{prop}\label{E0} Let $E_0$ be the maximal $p$-elementary abelian subgroup of the center of $P$ where $P$ is a Sylow $p$-subgroup of $G$.  Then if $\frak p_E\in\as H^*(G)$, $E_0\sub E$.\end{prop}
\begin{proof} Since $\res:H^*(G)\to H^*(P)$ is a $H^*(G)$ monomorphism, by \ref{non}, $$\mc Ass H^*(G)\sub\{\frak p^P_E\cap H^*(G)\vv \frak p^P_E\in\as H^*(P)\}.$$  Hence if $\frak p_E\in\as(G)$ there exists $E'\in \mc A(P)$
with $$\frak p_E=\frak p^P_{E'}\cap H^*(G)=\frak p_{E'}, \text{ and } \frak p^P_{E'}\sub \as\,H^*(P).$$ 
Since $\frak p_E=\frak p_{E'}$ they are conjugate by Quillen, \ref{Qu}.  That is we may always assume that up to conjugation if $\frak p_E\sub \as\, H^*(G)$, $\frak p^P_E\in\as H^*(P)$.  

Let $x\in Z(P)$ be an central element of order $p$ of $P$ and suppose $x\notin E$. Then $E'=(x)$ is not conjugate in $P$ to a subgroup of $E$ and so there exists $\tau\in\frak p^P_E\sub H^*(P)$ with $\tau\notin\frak p^P_{E'}$.  Since $\frak p^P_{E'}=\ker\{H^*(P)\to H^*(E')/\text{Rad}\}$,
$\text{res}^P_{E'}(\tau)\in H^*(E')$ is not nilpotent and therefore a non-zero divisor since $H^*(E')=P[\beta\,\bold u]\otimes \Lambda[\bold u]$. By Duflot's Theorem, \ref{Duf}, $\tau\in H^*(P)$ is a non zero-divisor. But $\tau\in\frak p^P_E=\text{ann}_{H^*(P)}(y)$ for some $0\neq y\in H^*(P)$.  Thus $\tau\,y=0$ and this is impossible.  Thus every central element of order $p$ is contained in $E$ and since $E_0$ is generated by elements of order $p$,  $E_0\sub E$.\end{proof}

\begin{thm}(Notbohm,\,\cite{N})\label{Not} $\dph H^*(G)=\min\{\dph H^*(C_G(E))\vv E\in\mathcal AG\}$.\end{thm}

\begin{thm}(Carlson,\,\cite{C}) Let $0\neq x\in H^*(G)$. If $\text{res}\,^G_E(x)=0$ for all $E\in\mathcal AG^s$, then $\dim H^*(G)/\text{ann}\,x< s$ and therefore there exists an associated prime $\frak p_E$ of $\dim<s$.\end{thm}

\begin{cor}(Poulsen,\,\cite{P})\label{Poul} $\min\{\dim \frak p_E\vv \frak p_E\in\as H^*(G)\}$ equals the maximum $s$ such that
$$ \prod_{E\in\mc AG^s}\res_{C_GE}:H^*(G)\to \prod_{E\in\mc AG^s}H^*(C_G(E))\} $$ is a monomorphism. That is, the maximum $s$ with $$\bigcap_{E\in\mc AG^s}\ker\{\res_{C_GE}\}=(0).$$
\end{cor}
 
\section*{Poset of Possible Associated Primes}

Recall $\mc AG$ the poset of conjugacy classes of elementary abelian $p$-subgroups of $G$
 where $E\leq E'$ if and only if $E$ is conjugate to a subgroup of $E'$.

Because of the Proposition \ref{E0} and Notbohm's Theorem, \ref{Not} we make the following definitions.  Let $P$ a Sylow $p$-subgroup of $G$ and $E_0$ the maximal elementary $p$-subgroup of the center of $P$.

\begin{defn}(1) PAP=$\{E\vv E_0\sub E\}\sub\mc AG$.
\par (2) AP is the subposet of PAP consisting of $\{E\in PAP\vv \p E\in\as H^*(G)\}$.
\par (3) $\mc B=\{E\in\mc AG\vv \dph H^*(G)=\dph H^*(C_GE)\}$.
\par (4) If $\mc C\sub \mc  AG$ is a subposet, let $\dim\,\mc C=\max\{\rk E\vv E\in\mc C\}$ and $\omega_\mc C=\min\{\rk E\vv E\in\mc C\}$.  If $\mc C=$AP then $\omega_\mc C=\omega_G$.
\par (5) If $\mc C\sub\mc AG$, denote $\{E\in\mc C\vv \rk E=s\}$ by $\mc C^s$.
\end{defn}

\begin{prop} Let $E\in\mc AG$.
\begin{enumerate}\item If $E\in\mc B$ then $\rk E\leq\dph H^*(G)\leq\omega _G$. That is $\dim\,\mc B\leq\dph H^*(G)$.
\item If $E\in\mc B$ and $E'\sub E$ then $E'\in\mc B$.
\item If $E_0$ is the maximal elementary subgroup of the center of a Sylow $p$-subgroup of $G$, then $E_0\in \mc B$.
\end{enumerate}\end{prop} 
\begin{proof}(1) If $\dim H^*(G)=\dph H^*(G)$ the result is trivial.  If $\dph H^*(G)<\rk E$ then $\dph H^*(C_GE)\geq \rk E>\dph H^*(G)$ by Duflot's theorem and therefore $E\notin\mc B$.
\par (2) $E'\sub E$ implies $E\sub C_GE\sub C_G(E')$. It is clear that $C_GE=C_{C_GE'}E$ and so by using Notbohm's theorem twice we have
\begin{equation*}\begin{split} \dph H^*(G)&\leq \dph H^*(C_GE')\leq\dph H^*(C_{C_GE'}E)\\&=\dph H^*(C_GE)=\dph H^*(G)\end{split}\end{equation*}
and therefore $E'\in\mc B$.
\par (3) $P\sub C_GE_0$. Since $[G:C_GE_0]$ is relatively prime to $p$, $H^*(G)$ is an $H^*(G)$ direct summand of $H^*(C_GE_0)$ and therefore $\dph H^*(C_GE_0)\leq\dph H^*(G)$.  But since $\dph H^*(G)\leq\dph H^*(C_GE)$ for all elementary abelian subgroups $E$ of $G$, by Notbohm's theorem we have $E_0\in\mc B$.\end{proof}

\begin{cor} If $E\in\text{AP}\cap\mc B$ then $\rk E=\omega_G=\dph H^*(G)$.\end{cor}
\begin{proof} If $E\in$ AP, 
$$\rk E\geq \omega_G\geq \dph H^*(G)=\dph H^*(C_GE)\geq \rk E.$$ \end{proof}




For any finite group $G$, let $z_G$ denote the Duflot bound for $G$, that is if $P$ is a Sylow $p$-subgroup of $G$
$$ z_G=\max\{\rk E\vv E\,\text{elementary abelian}\,\sub\,\text{center}\,P\}.$$

\noindent Duflot's result states $\dph H^*(G)\geq z_G$.

\begin{thm}\label{Sch} Let $G$ be a finite group and $E\in AP^{\omega_G}$, that is $\p E\in\as H^*(G)$ and $\rk E=\omega_G$. Then
$$z_{C_GE}=\omega_G=\omega_{C_G(E)}=\dph H^*(C_GE).$$ In particular, $C_GE$ satisfies the Connor conjecture for any $E\in AP^{\omega_G}$.\end{thm}
\begin{proof} There exists $\tau\in\p E\sub H^*(G)$ such that $\tau\notin \p{E'}$ for all $E'$ of rank $s$ and $E'$ not conjugate to $E$ by the corollary to Quillen's theorem.  Hence $\res_{E'}\tau$ is not in the radical of $H^*(E')$ and therefore not a zero-divisor.  Hence $\res_{C_GE'}\tau$ is not a zero divisor in $H^*(C_GE')$.  If $\p E=\an\,y$ then $\tau\,y=0$ and so $\text{res}_{C_GE'}(\tau\,y)=0$.  Since 
$\text{res}_{C_GE'}\tau$ is not a zero divisor in $H^*(C_GE')$ we have $\res_{C_GE'}y=0$.
On the other hand, \ref{Poul}, since
 
$$\omega_G=\max_E\{t\vv 0=\bigcap_{\rk E=t}\,ker\{\text{res}:H^*(G)\to H^*(C_GE)\},$$
we must have $\text{res}_{C_GE}\,y\neq 0$.  Let $J=\an_{H^*(C_GE)}\, \res_{C_GE}\,y\sub H^*(C_GE)$.  Since $H^*(C_GE)$ is \fg as an $H^*(G)$-module, we see 
$\dim\,\text{res}^{-1}J= \dim\,J$.  
But $\text{res}^{-1}J\supseteq\text{ann}_{H^*(G)}\,y=\p E$.  Since $J$ is an annihilator ideal, it is contained in an associated prime of $H^*(C_GE)$ and we have
$$\begin{aligned} \rk E=s&=\dim \p E\geq \dim\,\text{res}^{-1}J= \dim\,J\\&\geq\omega_{C_GE}\geq\dph H^*(C_GE)\geq z_{C_GE}\geq \rk E.\end{aligned}$$
$z_{C_GE}\geq \rk E$ since $E\sub Z(\text{Syl}_pCG)$.  Hence $s=z_{C_GE}=\omega_{C_GE}=\dph H^*(C_GE,k)$.

Since $AP^{\omega_G}\neq\emptyset$, it follows that for any $E\in AP^{\omega_G}$, $C_GE$ satifies the Connor conjecture.
\end{proof}

\begin{thm} Let $G$ be a finite group.  The following are equivalent.
\begin{enumerate}\item $G$ satisfies the Carlson conjecture.
\item For all $E\in\text{AP}^{\omega_G}$, that is with $\rk E=\omega_G$ and $\p E\in\as H^*(G)$ then $$\dph H^*(G)=\dph H^*(C_GE), \text{ i.e. } E\in\mc B. $$
\item There exists $\p E\in\text{Spec }H^*(G)$ with $$\rk E=\omega_G\text{ and }\dph H^*(G)=\dph H^*(C_GE).$$ That is, $\mc B^{\omega_G}\neq\emptyset$, equivalently $\dim\,\mc B=\omega_G$.
\end{enumerate}\end{thm}
\begin{proof}

\par (1) $\Rightarrow (2)$  Let $\p E\in\as H^*(G)$ with $\rk E=\omega_G$. By the previous theorem
$$ \omega_G=\omega_{C_G(E)}=\dph H^*(C_GE).$$

\noindent By Notbohm's theorem $\dph\, H^*(C_GE)\geq \dph H^*(G)$  and therefore if $\omega_G=\dph H^*(G)$ we have
$$\dph H^*(G)=\omega_G=\dph H^*(C_GE)\geq \dph H^*(G)$$ and therefore $\dph H^*(G)=\dph H^*(C_GE)$.

\medskip (2)$\Rightarrow$ (3) Since $\text{AP}^{\omega_G}\neq\emptyset$ by definition and $\mc B\cap\text{AP}\sub\mc B$.  

\medskip (3) $\Rightarrow$ (1)  $\rk E=\omega_G\geq\dph H^*(G)=\dph H^*(C_GE)\geq \rk E$ by Duflot's theorem.  Therefore $\omega_G=\dph H^*(G)$.
\end{proof}

\begin{cor} Let $G$ be a finite group.  Then $G$ satisfies the Carlson conjecture if and only if $\dph H^*(G)=\dph 
\prod_{\text{rk }E=\omega_G}H^*(C_GE)$.\end{cor}

\noindent Observations: 1) $H^*(G)$ is a subring of $\prod_{\text{rk }E=\omega_G}H^*(C_GE).$ $\prod_{\text{rk }E=\omega}H^*(C_GE)$ is finitely generated as an $H^*(G)$-module and the dimension of a finite product of rings is the maximum of the dimensions of the factors, we have $$\dim H^*(G)=\max\{\dim H^*(C_GE)\,\vert\, E\in \mc AG^\omega\}.$$
\par 2) It is true that if $R=\prod_1^nR_i$ is a finite product then $\dph R=\min\{\dph R_i\}$. It is not clear that 
$$\dph H^*(G)=\min\{\dph H^*(C_GE)\,\vert\, \text{rk }E=\omega_G\}.$$
although Magma seems to say it is so.

\begin{remark}1) $\text{AP}^{\omega_G}\subsetneq\mc B\subsetneq \mc A^{\omega_G}$. $G=SG[32,6]$ has (dim, depth)$\,(3,2)$. It has $6$ elementary abelian subgroups of rank $2$ with centralizers isomorphic to the the small groups $\{(16,11),(16,3),(8,5)\}$ with (dim,depth)$\,=\{(3,3),(3,2),(2,2)\}$. The rank $2$ elementary abelian subgroup $E$ ($L!10)$ whose centralizer is isomorphic to SG(8,5) is not an associated prime according to a calculation in Macaulay2. Therefore  $\text{AP}^{\omega_G}\subsetneq\mc B$. On the other hand the rank $2$ elementary abelian subgroup whose centalizer is isomorphic to SG(16,11) is not in $\mc B$.  Hence $\mc B\subsetneq \mc A^{\omega_G}$.

\par 2) It is true for $G$ Cohen-Macaulay that 
$\dph H^*(C_GE)=\dph H^*(G)$ and $\dim H^*(C_GE)=\dim H^*(G)$ for every $E\in \mc AG^\omega$ (see next corollary), these examples show that it is not true in general.  Moreover they show that although there exist $E\in\mc AG^\omega$ with $\dim H^*(C_GE)=\dim H^*(G)$ and may (probably) be possible to find $E'\in\mc AG^\omega$ with $\dph H^*(C_GE')=\dph H^*(G)$, it may or may not be possible to chose $E=E'$.

\par a) $G=SG(32,6)$ is of $(\text{dim,depth})=(3,2)$.  The centralizers of rank $2$ elementary abelian subgroups are isomorphic to $\{(8,5),(16,3),(16,11)\}$ of respective   $(\text{dim,depth})=\{(3,3),(3,2),(3,3)\}$.  

b) $G=SG(32,27)$ is of $(\text{dim,depth})=(4,3)$.  The centralizers of elementary abelian subgroups of rank $\omega_G$ are isomorphic to $\{(8,5),(16,14)\}$ of $(\text{dim,depth})=\{(3,3),(4,4)\}$ respectively.

c) $G=SG(32,43)$ is of $(\text{dim,depth})=(3,2)$.  The centralizers of rank $2$ elementary abelian subgroups are isomorphic to $\{(8,2),(8,5),(16,11)\}$ of respective $(\text{dim,depth})=\{(2,2),(3,3),(3,3)\}$.  

\end{remark}

\section*{Associated Primes of $H^*(C_GE)$}

In the following let $ res=\text{res}_{C_GE}:H^*(G)\to H^*(C_GE)$.

\begin{thm} Let $G$ be a finite group and $E\in AP^{\omega_G}$. Let $\frak p_E=\an y$, then
\begin{enumerate} 
\item $ ZD(H^*(C_GE))\sub\frak p^{C_GE}_E$ and therefore if $Q\in \text{Ass}\,H^*(C_GE)$, $Q\sub \frak p^{C_GE}_E$.
\item $res^{-1}\frak p^{C_GE}_E=\frak p^G_E$.
\item $\frak p^{C_GE}_E\in\as\,H^*(C_GE)$ and is therefore the unique maximal associated prime of $H^*(C_GE)$ (and of dimension $\omega_{C_GE}=\omega_G$).
\end{enumerate}
\end{thm}

\begin{proof} (1) Let $x\notin\frak p^{C_GE}_E$. Then $\res^{C_GE}_E\,x$ is not in the radical of $H^*(E)$ and therefore a non-zero divisor on $H^*(E)$.  By Duflot's Theorem $x$ is a non-zero divisor on $H^*(C_GE)$. Since the zero divisors are the union of the associated primes, the last statement follows. 
\par (2) From \ref{Sch} $\res\,y\neq 0$, $\an(\res\,y)\sub Q\sub\frak p^{C_GE}_E$ where $Q\in\as\,H^*(C_GE)$.  Hence
$$ \frak p_E\sub\res^{-1}\an(\res\,y)\sub\res^{-1}Q\sub\res^{-1}\frak p^{C_GE}_E. $$
Since $\frak p_E$ and $\res^{-1}\frak p^{C_GE}_E$ are primes of $H^*(G)$ of the same dimension, namely the rank of $E$, they must be equal.
\par (3) Since $H^*(C_GE)$ is \fg over $H^*(G)$, for any ideal $J\sub H^*(C_GE)$, $\dim\,\res^{-1}J=\dim\,J$ by the Going-Up Theorem.  Therefore $\dim\,Q=\dim\,\frak p^{C_GE}_E$ and since they are both primes with $Q\sub\frak p^{C_GE}_E$ they are equal.
\end{proof}

Every associated prime of a ring $R$ is the annihilator of an element $u\in R$ or equivalently the radical of the annihilator of some element $v\in R$.  Although we cannot determine $u\in H^*(C_GE)$ with $\frak p^{C_GE}_E=\an u$ we do have
\begin{thm} $\frak p^{C_GE}_E=\text{rad}\,(\an(\res\,y))$. \end{thm}
\begin{proof} This will follow from a sequence of observations and propositions.  Let $P=\frak p^{C_GE}_E$.  Recall $J=\an\,\res\,y\sub P$ with $P\in \text{Ass}\,H^*(C_GE)$ and $\dim J=\dim P$.
\medskip
\par (1) If $(R,M)$ is a local or *-local ring with $\dim R=0$ then $M^s=(0)$ for some $s>0$.
\begin{proof} $R$ is Artinian so $M^s=M^{s+1}=M\, M^s$ for some $s>0$ and therefore by Nakayama's Lemma, $M^s=(0)$.\end{proof}

\par (2) $(R,M)$ local or *-local and $J\sub M$, homogeneous in the *-local case, then $\dim R/J=0$ if and only if $M^s\sub J$ for some $s>0$.
\begin{proof} $(\bar R,\bar M)=(R/J, M/J)$ is local or *-local.  If $\dim\bar R=0$ then $\bar M^s=(\bar 0)$, i.e. $M^s\sub J$.  If $M^s\sub J$ then $V(M)={M}\sub V(J)\sub V(M^s)=\cup_sV(M)={M}$.  Hence $V(J)={M}$. Here $V(I)$ denotes the homogenous spectrum of a homogeneous ideal. Therefore $\dim J=\dim R/J=0$. \end{proof}

Now let $R=H^*(C_GE)$ and let $R_{(P)}$ denote the homogeneous localization of $R$, i.e., only invert the homogeneous elements not in $P$.  Let $\pi:R\to R_{(P)}$ denote the natural map.

\par (3) If $I$ is a homogeneous annihilator ideal contained in $P\sub R$, then $I=\pi^{-1}I_{(P)}$.
\begin{proof} If $S$ denotes the set of homogeneous elements not in $P$, then  
$$ \pi^{-1}I_{(P)}=\{a\in R\vv \text{there exists } s\in S\text{ with }s\,a\in I \}.$$ If $I=\an\,U\sub P$ and $s$ is homogeneous not in $P$, $a\in R$ and $s\,a\in I$, then for all $u\in U$, $(s\,a)u=s(a\,u)=0$ and hence $a\,u=0$ since $s$ is a non-zro divisor. That is $a\in I$ and $I=\pi^{-1}I_{(P)}$. \end{proof}

\par (4) $J=\an\,\res\,y\sub P$ and $J$, $P$ are homogeneous of same dimension. Every homogeneous prime of $R_{(P)}$ is $Q_{(P)}$ for some unique homogeneous prime $Q\sub P$ of $R$.  Therefore if $J_{(P)}\sub Q_{(P)}\sub P_{(P)}$ which implies
$$ J=\pi^{-1}J_{(P)}\sub Q=\pi^{-1} Q_{(P)}\sub P=\pi^{-1} P_{(P)}.$$
Hence $Q=P$, $Q_{(P)}=P_{(P)}$ and $\dim J_{(P)}=0$.  Therefore there exists $s\geq 1$ with $(P_{(P)})^s\sub J_{(P)}$.
Hence $$ P^s\sub \pi^{-1}(P^s)_{(P)}= \pi^{-1}(P_{(P)})^s\sub \pi^{-1}J_{(P)}=J $$
and $P\sub\text{rad}\,J$. Since the converse is obvious $P=\text{rad}\,(\an\,\res\,y)$.
\end{proof}

\end{document}